\documentclass[12pt, oneside, psamsfonts]{amsart}

\newif\ifPDF
\ifx\pdfoutput\undefined\PDFfalse
\else \ifnum \pdfoutput > 0 \PDFtrue
        \else \PDFfalse
        \fi
\fi

\usepackage[centertags]{amsmath}
\usepackage{amsfonts}
\usepackage{mathrsfs}
\usepackage{textcomp}
\usepackage{amssymb}
\usepackage{amsthm}
\usepackage{newlfont}
\usepackage[all]{xy}

\ifPDF
  \usepackage[pdftex]{color, graphicx}
  \usepackage[pdftex, bookmarks, colorlinks]{hyperref}
  \hypersetup{colorlinks=false}

\else
  \usepackage{color}
  \usepackage[dvips]{graphicx}
  \usepackage[dvips]{hyperref}
\fi

\usepackage[scale=0.8]{geometry}

\usepackage[pagewise, mathlines, displaymath]{lineno}

\newtheorem{thm}{Theorem}[section]
\newtheorem{cor}[thm]{Corollary}
\newtheorem{lem}[thm]{Lemma}

\theoremstyle{definition}
\newtheorem{defn}[thm]{Definition}
\theoremstyle{remark}
\newtheorem{rem}[thm]{Remark}
\newtheorem{example}[thm]{Example}
\numberwithin{equation}{section}

\newcommand{\norm}[1]{\left\Vert#1\right\Vert}
\newcommand{\abs}[1]{\left\vert#1\right\vert}

\newcommand{\Real}{\mathbb R}
\newcommand{\Int}{\mathbb Z}
\newcommand{\Comp}{\mathbb C}

\newcommand{\eps}{\varepsilon}

\begin{document}


\title{Comparison radius and mean topological dimension: $\Int^d$-actions}

\author{Zhuang Niu}
\address{Department of Mathematics, University of Wyoming, Laramie, Wyoming, USA, 82071}
\email{zniu@uwyo.edu}

\thanks{The research is supported by the NSF grant DMS-1800882.}


\begin{abstract}
Consider a minimal free topological dynamical system $(X, T, \Int^d)$. It is shown that the comparison radius of the crossed product C*-algebra $\mathrm{C}(X) \rtimes \Int^d$ is at most the half of the mean topological dimension of $(X, T, \Int^d)$. As a consequence, the C*-algebra $\mathrm{C}(X) \rtimes \Int^d$ is classifiable if $(X, T, \Int^d)$ has zero  mean dimension.
\end{abstract}

\maketitle

\section{Introduction}
Consider a topological dynamical system $(X, \sigma, \Gamma)$, where $X$ is a compact Hausdorff space and $\Gamma$ is a discrete amenable group. The mean (topological) dimension of $(X, \sigma, \Gamma)$, denoted by $\mathrm{mdim}(X, \sigma, \Gamma)$, was introduced by Gromov (\cite{Gromov-MD}), and then was developed and studied systematically by Lindenstrauss and Weiss (\cite{Lindenstrauss-Weiss-MD}). It is a numerical invariant, taking value in $[0, +\infty]$, to measure the complexity of $(X, \sigma, \Gamma)$ in terms of dimension growth with respect to partial orbits.

On the other hand, for a general C*-algebra $A$, the comparison radius,  introduced by Toms (\cite{RC-Toms}) and denoted by $\mathrm{rc}(A)$, plays a role as the dimension growth of $A$. A typical example is that the comparison radius of $\mathrm{M}_n(\mathrm{C}(X))$ is at most $\frac{1}{2} \frac{\mathrm{dim}(X)}{n}$, which is the half of the dimension ratio of $\mathrm{M}_n(\mathrm{C}(X))$. 

In this paper, it is shown that if $(X, T, \Int^d)$ is a free minimal $\Int^d$-action on a separable compact Hausdorff space $X$, then 
\begin{equation}\label{rc-md-comp}
\mathrm{rc}(\mathrm{C}(X) \rtimes \Int^d) \leq \frac{1}{2} \mathrm{mdim}(X, T, \Int^d),
\end{equation} 
where $\mathrm{C}(X) \rtimes \Int^d$ is  the crossed product C*-algebra associated to $(X, T, \Int^d)$. The argument is in the line with \cite{Niu-MD-Z}. That is, the dynamical system $(X, T, \Int^d)$ is shown to have a Cuntz comparison property on open sets and to have the Uniform Rokhlin Property; and then \eqref{rc-md-comp} follows from Theorem 8.8 of \cite{Niu-MD-Z}. The adding-one-dimension and going-down argument of \cite{GLT-Zk} play a crucial role in the proof of the Cuntz comparison property and the (URP).


\section{Notation and Preliminaries}

\subsection{Topological Dynamical Systems}

In this paper, one only considers $\Int^d$ actions on a compact separable Hausdorff space $X$.
\begin{defn}
Consider a topological dynamical system $(X, T, \Int^d)$. 
A closed set $Y \subseteq X$ is said to be invariant if $T^n(Y) = Y$, $n \in\Int^d$, and $(X, T, \Int^d)$ is said to be minimal if $\varnothing$ and $X$ are the only invariant closed subsets. The dynamical system $(X, T, \Int^d)$ is free if for any $x\in X$, $\{n\in \Int^d: T^n(x) = x\} = \{0\}$.
\end{defn}

\begin{rem}
The dynamical system $(X, T, \Int^d)$ is induced by $d$ commuting homeomorphisms of $X$, and vise versa.
\end{rem}

\begin{defn}
A Borel measure $\mu$ on $X$ is invariant under the action $\sigma$ if $\mu(E) = \mu(T^n(E))$, for any $n\in\Int^d$ and any Borel set $E\subseteq X$. Denote by $\mathcal M_1(X, T, \Int^d)$ the collection of all invariant Borel probability measures on $X$. It is a Choquet simplex under the weak* topology.
\end{defn}

\begin{defn}[see \cite{Gromov-MD} and \cite{Lindenstrauss-Weiss-MD}]
Consider a topological dynamical system $(X, T, \Int^d)$, and let $E$ be a subset of $X$. The orbit capacity of $E$ is defined by
$$\mathrm{ocap}(E):=\lim_{N\to\infty}\frac{1}{N^d}\sup_{x\in X}\sum_{n\in\{0, 1, ..., N-1\}^d} \chi_E(T^n(x)),$$
where $\chi_E$ is the characteristic function of $E$. The limit always exists.
\end{defn}

\begin{defn}[see \cite{Lindenstrauss-Weiss-MD}]
Let $\mathcal U$ be an open cover of $X$. Define
$$D(\mathcal U)=\min\{\mathrm{ord}(\mathcal V): \mathcal V \preceq U\},$$ where $\mathcal V=-1 + \sup_{x\in X}\sum_{V\in\mathcal V} \chi_V(x)$.

Consider a topological dynamical system $(X, T, \Int^d)$. Then the topological mean dimension of $(X, T, \Int^d)$ is defined by
$$\mathrm{mdim}(X, T, \Int^d):=\sup_{\mathcal U}\lim_{N\to\infty}\frac{1}{N^d}D(\bigvee_{n\in\{0, 1, ..., N-1\}^d} T^{-n}(\mathcal U)),$$
where $\mathcal U$ runs over all finite open covers of $X$.
\end{defn}

\begin{rem}
It follows from the definition that if $\mathrm{dim}(X) < \infty$, then  $\mathrm{mdim}(X, T, \Int^d)=0$;
By \cite{Lindenstrauss-Weiss-MD}, if $(X, T, \Int^d)$ has at most countably many ergodic measures, then $\mathrm{mdim}(X, T, \Int^d)=0$;   and by \cite{Lind-MD}, if $(X, T, \Int^d)$ has finite topological entropy, then $\mathrm{mdim}(X, T, \Int^d)=0$.
\end{rem}

\subsection{Crossed product C*-algebras}
Consider a topological dynamical system $(X, T, \Int^d)$. 
Then the crossed product C*-algebra $\mathrm{C}(X)\rtimes \Int^d$ is the universal C*-algebra 
$$A=\textrm{C*}\{f, u_n;\ u_n fu_n^*=f\circ T^n,\ u_mu^*_n=u_{m-n},\ u_0=1,\  f\in \mathrm{C}(X),\ m, n\in \Int^d \}.$$
The C*-algebra $A$ is nuclear, and if $T$ is minimal, the C*-algebra $A$ is simple. Moreover, the simplex of tracial states of $\mathrm{C}(X) \rtimes_\sigma \Gamma$ is canonically homeomorphic to the simplex of the invariant probability measures of $(X, T, \Int^d)$.

\subsection{Cuntz comparison of positive elements of a C*-algebra}

\begin{defn}
Let $A$ be a C*-algebra, and let $a, b\in A^+$. Then we say that $a$ is Cuntz subequivalent to $b$, denote by $a \precsim b$, if there are $x_i$, $y_i$, $i=1, 2, ...$, such that $$\lim_{n\to\infty} x_iby_i = a,$$ and we say that $a$ is Cuntz equivalent to $b$ if $a\precsim b$ and $b \precsim a$.

Let $\tau: A \to\Comp$ be a trace. Define the rank function
$$\mathrm{d}_\tau(a):=\lim_{n\to\infty}\tau(a^{\frac{1}{n}})=\mu_\tau(\mathrm{sp}(a)\cap(0, +\infty)),$$ where $\mu_\tau$ is the Borel measure induced by $\tau$ on the spectrum of $a$.  It is well known that
$$\mathrm{d}_\tau(a) \leq \mathrm{d}_\tau(b),\quad \textrm{if $a\precsim b$}.$$
\end{defn}
\begin{example}
Consider $h\in \mathrm{C}(X)^+$ and let $\mu$ be a probability measure on $X$. Then
$$\mathrm{d}_{\tau_\mu} = \mu(f^{-1}(0, +\infty)),$$
where $\tau_\mu$ is the trace of $\mathrm{C}(X)$ induced by $\mu$. 

Let $f, g\in\mathrm{C}(X)$ be positive elements. Then $f$ and $g$ are Cuntz equivalent if and only if $f^{-1}(0, +\infty)= g^{-1}(0, +\infty)$. That is, their equivalence classes are determined by their open support. On the other hand, for each open set $E \subseteq X$, pick a continuous function $$\varphi_{E}: X \to [0, +\infty)\quad \textrm{such that} \quad E = \varphi_E^{-1}(0, +\infty).$$ For instance, one can pick $\varphi_E(x) = d(x, X\setminus E)$, where $d$ is a compatible metric on $X$. This notation will be used throughout this paper. Note that the Cuntz equivalence class of $\varphi_E$ is independent of the choice of individual function $\varphi_E$.
\end{example}

\begin{defn}\label{defn-p-cut}
Let $a\in A^+$, where $A$ is a C*-algebra, and let $\eps>0$. Define
$$(a - \eps)_+ =f(a) \in A,$$
where $f(t)=\max\{t-\eps, 0\}$.
\end{defn}
A frequently used fact on the Cuntz comparison is the following.
\begin{lem}[Section 2 of \cite{RorUHF2}]
Let $a, b$ be positive elements of a C*-algebra $A$. Then $a \precsim b$ if and only if $(a-\eps)_+ \precsim b$ for all $\eps >0$.
\end{lem}

\begin{defn}[Definition 6.1 of \cite{RC-Toms}]
Let $A$ be a C*-algebra. Denote by $\mathrm{M}_n(A)$ the C*-algebra of $n\times n$ matrices over $A$. Regard $\mathrm{M}_n(A)$ as the upper-left conner of $\mathrm{M}_{n+1}(A)$, and denote by $$\mathrm{M}_\infty(A) = \bigcup_{n=1}^\infty \mathrm{M}_n(A),$$ the algebra of all finite matrices over $A$.

The radius of comparison of a unital C*-algebra $A$, denoted by $\mathrm{rc}(A)$, is the infimum of the set of real numbers $r > 0$ such that  if $a, b\in (\mathrm{M}_\infty(A))^+$ satisfy
$$\mathrm{d}_\tau(a) + r < \mathrm{d}_\tau(b),\quad \tau\in\mathrm{T}(A),$$ then $a \precsim b$, where $\mathrm{T}(A)$ is the simplex of tracial states. (In \cite{RC-Toms}, the radius of comparison is defined in terms of quasitraces instead of traces; but since all the algebras considered in this note are nuclear, by \cite{Haagtrace}, any quasitrace actually is a trace.)
\end{defn}
\begin{example}
Let $X$ be a compact Hausdorff space. Then 
\begin{equation}\label{comp-dim-control}
\mathrm{rc}(\mathrm{M}_n(\mathrm{C}(X)))\leq \frac{1}{2}\frac{\mathrm{dim}(X) - 1}{n},
\end{equation} 
where $\mathrm{dim(X)}$ is the topological covering dimension of $X$ (a lower bound of $\mathrm{rc}(\mathrm{C}(X))$ in terms of cohomological dimension is given in \cite{EN-RC}).
\end{example}

The main result of this paper is a dynamical version of \eqref{comp-dim-control}; that is,
$$\mathrm{rc}(\mathrm{C}(X) \rtimes \Int^d) \leq \frac{1}{2}\mathrm{mdim}(X, T, \Int^d)$$
if $(X, T, \Int^d)$ is minimal and free (Corollary \ref{rc-mdim}).

\section{Adding one dimension, going-down argument, $R$-boundary points, and $R$-interior points}

Adding-one-dimension and going-down argument are introduced in \cite{GLT-Zk}, and they play a crucial role in this paper. Let us first take a brief review. Consider a minimal system $(X, T, \Int^d)$. Pick open sets $U' \subseteq U \subseteq X$ with $\overline{U'} \subseteq U$, and a continuous function $\varphi: X \to [0, 1] $ such that 
$$\varphi|_{U'}=1\quad\mathrm{and} \quad \varphi|_{X\setminus U} = 0.$$
Then there exist natural numbers $M\leq L$ such that
\begin{enumerate}
\item if $\varphi(x)>0$ for some $x\in X$, then $\varphi(T^n(x))=0$ for all nonzero $n\in\Int^k$ with $\abs{n}\leq M$; and
\item for any $x\in X$, there is $n\in\Int^d$ with $\abs{n}\leq L$ such that $\varphi(T^n(x))=1$.
\end{enumerate}

Pick $x\in X$. Following from \cite{GLT-Zk}, one considers the set
$$\{(n, \frac{1}{\varphi(T^n(x))}): n \in\Int^d,\ \varphi(T^n(x))\neq 0\} \subseteq \Real^{d+1},$$
and defines the Voronoi  cell $V(x, n) \subseteq \Real^{d+1}$ with center $(n, \frac{1}{\varphi(T^n(x))})$ by 
$$V(x, n) = \left\{\xi\in \Real^{k+1}: \norm{\xi-(n, \frac{1}{\varphi(T^n(x))})} \leq \norm{\xi-(m, \frac{1}{\varphi(T^m(x))})}, \forall m\in\Int^d \right\},$$
where $\norm{\cdot}$ is the $\ell^2$-norm on $\Real^{d+1}$.
If $\varphi(T^n(x))=0$, then put
$$V(x, n)=\varnothing.$$
One then has a tiling
$$\Real^{d+1} = \bigcup_{n\in\Int^{d}} V(x, n).$$

Pick $H> (L+ \sqrt{d})^2$. For each $n\in \Int^{d}$, define $$W_H(x, n)= V(x, n) \cap (\Real^d \times \{-H\}),$$ and one has a tiling
$$\mathcal W_H: \Real^{d} = \bigcup_{n\in\Int^{d}} W(x, n).$$
The following are some basic properties of this construction, and the proofs can be found in \cite{GLT-Zk}.
\begin{lem}[Lemma 4.1 of \cite{GLT-Zk}]\label{tiling-properties}
With the construction above, one has
\begin{enumerate}
\item $\mathcal W_H$ is continuous on $x$ in the following sense: Suppose that $W(x, n)$ has non-empty interior. For any $\eps>0$, if $y\in X$ is sufficiently close to $x$, then the Hausdorff distance between $W_H(x, n)$ and $W_H(y, n)$ are smaller than $\eps$.
\item $\mathcal W_H$ is $\Int^d$-equivariant: $W_H(T^m(x), n-m) = - m + W_H(x, n)$.
\item If $\varphi(T^n(x)) > 0$, then $$B_{\frac{M}{2}}(n, \frac{1}{\varphi(T^n(x))}) \subseteq V(x, n).$$
\item If $W_H(x, n)$ is non-empty, then $$1\leq \frac{1}{\varphi(T^n(x))} \leq 2.$$
\item If $(a, -H) \in V(x, n)$, then $$\norm{a - n} < L+\sqrt{d}.$$
\end{enumerate}
\end{lem}

Moreover, if one considers different horizontal cuts (at levels $-sH$ and $-H$ for some $s>1$), one has the following lemma.
\begin{lem}[Lemma 4.1(4) of \cite{GLT-Zk} and its proof]\label{cut-down-large}
Let $s>1$ and $r>0$. One can choose $M$ sufficiently large such that if $(a, -sH) \in V(x, n)$, then $$B_r(\frac{a}{s} + (1 - \frac{1}{s})n) \subseteq W_H(x, n)$$
and
$$\norm{\frac{a}{s} + (1 - \frac{1}{s})n - (a + \frac{(s-1)H}{sH+t}(n-a))} \leq \frac{4}{L+\sqrt{d}},$$
where $t= \frac{1}{\varphi(T^n(x))}$ and $\norm{\cdot}$ is the $\ell^2$-norm on $\Real^d$.
\end{lem}

\begin{defn}
Note that the point $(a + \frac{(s-1)H}{sH+t}(n-a), -H)$ is the image of $(a, -sH)$ in the plane $\Real^d \times \{-H\}$ under the projection towards the center $(n, t)$. Let us call $a + \frac{(s-1)H}{sH+t}(n-a)$ the $H$-projective image of $a$ (with the center $(n, t)$). 
\end{defn}

The following is a lemma on convex bodies in $\Real^d$, and the author is in debt to Tyrrell McAllister for the  discussions.
\begin{lem}\label{convex-small-bd-N-Z}
Consider $\Real^d$. For any $\eps>0$ and any $r>0$, there is $N_0>0$ such that if $N\geq N_0$, then for any convex body $V \subseteq \Real^d$, one has
$$\frac{1}{N^d}\abs{\{n\in \Int^d: \mathrm{dist}(n, \partial V) \leq r,\ n\in I_N \}} < \eps,$$
where $I_N = [0, N]^d$.
\end{lem}

\begin{proof}
Pick $N_0$ sufficiently large such that 
$$2 \frac{\mathrm{vol}(\partial_{r+\sqrt{d}}(I_N))}{\mathrm{vol}(I_N)} < \eps,\quad N>N_0,$$ where $\partial_E(K)$ denotes the $E$-neighbourhood of the boundary of a convex body $K$.  Then, this $N_0$ satisfies the conclusion of the Lemma.

Indeed, for any $N\geq N_0$, denote by $\partial^+_{r+\sqrt{d}}(V \cap I_N)$ the outer $(r+\sqrt{d})$-neighborhood of the convex body  $V \cap I_N$, and it follows from Steiner formula (see, for instance, (4.1.1) of \cite{Schneider-book}) that
$$\mathrm{vol}(\partial^+_{r+\sqrt{d}}(V \cap I_N)) =  \sum_{j=1}^d C_d^j W_j(V\cap I_N)(r+\sqrt{d})^j,$$
where $W_j(V \cap I_N)$ is the $j$-th quermassintegral of $V\cap I_N$. Since the quermassintegrals $W_j$, $j=1, ..., d$, are monotonic (see, for instance, Page 211 of \cite{Schneider-book}), one has $$W_j(V \cap I_N) \leq W_j(I_N),\quad j=1, 2, ..., d,$$ and hence 
\begin{eqnarray*}
\mathrm{vol}(\partial^+_{r+\sqrt{d}}(V \cap I_N)) & = & \sum_{j=1}^d C_d^j W_j(V\cap I_N)(r+\sqrt{d})^j \\ 
& \leq & \sum_{j=1}^d C_d^j W_j(I_N)(r+\sqrt{d})^j \\
& = & \mathrm{vol}(\partial^+_{r + \sqrt{d}}(I_N)).
\end{eqnarray*}
Since $\mathrm{vol}(\partial_{r+\sqrt{d}}(V\cap I_N)) \leq 2\mathrm{vol}(\partial^+_{r+\sqrt{d}}(V\cap I_N))$, one has
$$\frac{\mathrm{vol}(\partial_{r+\sqrt{d}}(V\cap I_N))}{\mathrm{vol}(I_N)} \leq 2\frac{\mathrm{vol}(\partial^+_{r+\sqrt{d}}(V\cap I_N))}{\mathrm{vol}(I_N)} \leq 2 \frac{\mathrm{vol}(\partial_{r+\sqrt{d}}(I_N))}{\mathrm{vol}(I_N)} < \eps. $$
On the other hand, note that
$$  \abs{\{n\in \Int^d: \mathrm{dist}(n, \partial V) \leq r,\ n\in I_N \}} \leq \mathrm{vol}(\partial_{r+\sqrt{d}}(V\cap I_N)), $$ and hence
$$\frac{1}{N^d}\abs{\{n\in \Int^d: \mathrm{dist}(n, \partial V) \leq r,\ n\in I_N \}} \leq \frac{\mathrm{vol}(\partial_{r+\sqrt{d}}(V\cap I_N))}{\mathrm{vol}(I_N)} < \eps,$$
as desired.
\end{proof}

\begin{defn}
Consider a continuous function $X \ni x \mapsto \mathcal W(x)$ with $\mathcal W(x)$ a $\Real^d$-tiling. For each $R \geq 0$, a point $x\in X$ is said to be an $R$-interior point if $\mathrm{dist}(0, \partial \mathcal W(x)) > R$, where $\partial \mathcal W(x)$ denotes the union of the boundaries of the tiles of $\mathcal W$. Note that, in this case, the origin $0\in\Real^d$ is an interior point of a (unique) tile of $\mathcal W(x)$. Denote this tile by $\mathcal W(x)_0$, and denote the set of $R$-interior points by $\iota_R(\mathcal T)$.

Otherwise (if $\mathrm{dist}(0, \partial \mathcal W(x)) \leq R$), the point $x$ is said to be an $R$-boundary point. Denote by $\beta_R(\mathcal T)$ the set of $R$-boundary points. 

Note that $\beta_R(\mathcal T)$ is closed and $\iota_R(\mathcal T)$ is open.
\end{defn}

\begin{lem}\label{pre-two-towers}
Let $(X, T, \Int^d)$ be a minimal free dynamical system.

Fix $s \in (1, 2)$. Let $R_0> 0$ and $\eps>0$ be arbitrary. Let $N>N_0$, where $N_0$ the constant of Lemma \ref{convex-small-bd-N-Z} with respect to $\eps$ and $2R_0+4+\sqrt{d}/2$, and let $R_1 > \max\{R_0, N\sqrt{d}\}$. 

Then  $M$ can be chosen large enough such that there exist a finite open cover
$$U_1 \cup U_2 \cup \cdots \cup U_K \supseteq \beta_{R_0}(\mathcal W_{sH}),$$ and $n_1, n_2, ..., n_K \in \Int^d$ such that
\begin{enumerate}
\item\label{insert-prop-1} $T^{n_i}(U_i) \subseteq \iota_{R_1}(\mathcal W_H) \subseteq \iota_0(\mathcal W_H)$, $i=1, 2, ..., K$,
\item\label{insert-prop-2} the open sets $$T^{n_i}(U_i),\quad i=1, 2, ..., K,$$ can be grouped as 
$$
\left\{  
\begin{array}{l}
T^{n_{1}}(U_{1}), ..., T^{n_{s_1}}(U_{s_1}), \\
T^{n_{s_1+1}}(U_{s_1+1}), ..., T^{n_{s_2}}(U_{s_2}), \\
\cdots \\
T^{n_{s_{m-1}+1}}(U_{s_{m-1}+1}), ..., T^{n_{s_m}}(U_{s_m}),
\end{array}
\right.
$$
with $m \leq (\lfloor2\sqrt{d}\rfloor + 1)^d$, such that the open sets in each group are mutually disjoint, 
\item\label{insert-prop-3} for each $x\in \iota_0(\mathcal W_H)$ and each $c\in \mathrm{int}(\mathcal W_H(x)_0) \cap \Int^d$ with $\mathrm{dist}(c, \partial\mathcal W_H)>N\sqrt{d}$, one has
         $$\frac{1}{N^d} \abs{\left\{n \in \{0, 1, ..., N-1\}^d: T^{c+n}(x) \in \bigcup_{i=1}^K T^{n_i}(U_i)\right\}} < \eps.$$         
\end{enumerate}
\end{lem}

\begin{proof}

By Lemma 4.1(4) of \cite{GLT-Zk} (see Lemma \ref{cut-down-large}), one can choose $U'\subseteq U$ and $\varphi$ such that $M$ is sufficiently large so that for a fixed $H>(L+\sqrt{d})^2$, if $(a, -sH) \in V(x, n)$ for some $a\in\Real^d$, then 
$$B_{R_1+2R_0 + 1+ \frac{\sqrt{d}}{2}}(\frac{a}{s} + (1-\frac{1}{s})n) \times \{-H\} \in V(x, n)$$  
and 
\begin{equation}\label{close-to-bd-1}
\norm{\frac{a}{s}+(1-\frac{1}{s})n - (a+  \frac{(s-1)H}{sH+t}(n-a))} \leq\frac{4}{L+\sqrt{d}} < 4,
\end{equation}
where $t = \frac{1}{\varphi(T^n(x))}$, and $a+  \frac{(s-1)H}{sH+t}(n-a)$ is the $H$-projective image of $a$.


For each $n\in \Int^d$, define
$$U_n = \{x \in X: \mathrm{dist}(0, \partial W_{sH}(x, n)) < 2R_0,\ 
\mathrm{int}W_{sH}(x, n) \neq \varnothing\}.$$ Note that $U_n$ is open. For the same $n$, one also picks $h_n\in \Int^d$ such that 
\begin{equation}\label{pert-n}
\norm{(1-\frac{1}{s})n - h_n} \leq \frac{\sqrt{d}}{2}.
\end{equation}

For each $x \in U_n$, 
there is $a \in \partial W_{sH}(x, n)\subseteq \Real^d$ with $$\norm{a} < 2R_0. $$ By the choice of $M$ (hence $H$), one has
\begin{equation}\label{large-H}
 B_{R_1+2R_0 + 1+ \frac{\sqrt{d}}{2}}(\frac{a}{s} + (1-\frac{1}{s})n) \subseteq W_{H}(x, n).
\end{equation} 
Since
\begin{equation}\label{close-to-bd-2}
\norm{h_n - (\frac{a}{s} + (1-\frac{1}{s})n)} \leq \norm{\frac{a}{s}} + \norm{(1-\frac{1}{s})n - h_n} < 2R_0 +\frac{\sqrt{d}}{2},
\end{equation} 
by \eqref{large-H}, one has
$$B_{R_1+1}(h_n) \subseteq W_H(x, n),$$ which implies
\begin{equation}\label{contain-R-ball}
B_{R_1}(0) \subset B_{R_1+1}(0) \subseteq - h_n +  W_H(x, n) = W_H( T^{h_n}(x), n  - h_n).
\end{equation}
In particular, $T^{h_n}(x) \in \iota_{R_1}(\mathcal W_H)$, which implies $$T^{h_n}(U_n) \subseteq \iota_{R_1}(\mathcal W_H),$$ and this shows Property (\ref{insert-prop-1}).

Note that by \eqref{close-to-bd-1} and \eqref{close-to-bd-2},
\begin{equation}\label{image-nbhd}
\norm{h_n - (a+\frac{(s-1)H}{sH+t}(n-a))} < 2R_0+4+\frac{\sqrt{d}}{2}.
\end{equation}
Since $a \in \partial W_{sH}(x, n)$, this implies that $h_n$ is in the $(2R_0+4+\frac{\sqrt{d}}{2})$-neighbourhood of the the $H$-projective image of $\partial W_{sH}(x, n)$ (with respect to $(n, t)$).

On the other hand, if $x\in \beta_{R_0}(\mathcal W_{sH})$, then $\mathrm{dist}(0, \partial W_{sH}(x, n)) \leq R_0$ for some $n\in\Int^d$ with $\mathrm{int}(W_{sH}(x, n)) \neq \varnothing$, which implies that $x \in U_n$. Therefore, $\{U_n: n\in \Int^d\}$ form an open cover of $\beta_{R_0}(\mathcal W_{sH})$. Since $\beta_{R_0}(\mathcal W_{sH})$ is a compact set, there is a finite subcover
$$U_{n_1}, U_{n_2}, ..., U_{n_K}.$$
(In fact, $\{U_n: \norm{n}<L+\sqrt{d}+ 2R_0\}$ already covers $\beta_{R_0}(\mathcal W_{sH})$ by (5) of Lemma \ref{tiling-properties}.)

Assume that $n_i$ and $n_j$ satisfy $$T^{h_{n_i}} (U_{n_i}) \cap T^{h_{n_j}} (U_{n_j}) \neq \varnothing.$$ Then there are $x_i\in U_{n_i}$ and $x_j\in U_{n_j}$ with $$T^{h_{n_i}}(x_i) = T^{h_{n_j}}(x_j).$$
Since $x_i \in U_{n_i}$ and $x_j \in U_{n_j}$, by \eqref{contain-R-ball}, one has that 
$$B_R(0) \subseteq W_{H}(T^{h_{n_i}}(x_i), n_i-h_{n_i})$$
and
\begin{eqnarray*}
B_R(0) & \subseteq & W_{H}(T^{h_{n_j}}(x_j), n_j-h_{n_j})\\
 & = & W_H(T^{h_{n_i}}(x_i), n_j-h_{n_j}).
\end{eqnarray*}
Therefore, $n_i-h_{n_i} = n_j-h_{n_j}$, and
$$n_i - n_j = h_{n_i}-h_{n_j}.$$
Together with \eqref{pert-n}, one has
\begin{eqnarray*}
\norm{n_i - n_j} & = & \norm{h_{n_j} - h_{n_j}} \\
& \leq & (1-\frac{1}{s})\norm{n_i - n_j} +  \sqrt{d} \\
& < & \frac{1}{2}\norm{n_i - n_j} +  \sqrt{d},
\end{eqnarray*}
and hence $$\norm{n_i - n_j} < 2\sqrt{d}.$$
Note that the set $\Int^d$ can be divided into $(\lfloor2\sqrt{d}\rfloor+1)^d$ groups $(\Int^d)_1$, ..., $(\Int^d)_{(\lfloor2\sqrt{d}\rfloor + 1)^d}$ such that any pair of elements inside each group has distance at least $2\sqrt{d}$, and therefore $$T^{h_n}(U_n) \cap T^{h_{n'}}(U_{n'}) = \varnothing,\quad n, n' \in (\Int^d)_m,\ m=1, ..., (\lfloor2\sqrt{d}\rfloor + 1)^d.$$ Then group $U_{n_1}, ..., U_{n_K}$ as
$$\{U_{n_i}: i=1, ..., K,\ n_i\in (\Int^d)_1\}, ..., \{U_{n_i}: i=1, ..., K,\ n_i\in (\Int^d)_{(\lfloor2\sqrt{d}\rfloor + 1)^d}\},$$ and this shows Property (\ref{insert-prop-2}).

Let $x\in \iota_0(\mathcal W_H)$ (so that $\mathcal W_H(x)_0$ is well defined). Write $$\mathcal W_H(x)_0 = W_H(x, n(x)) = V(x, n(x)) \cap (\Real^d \times \{-H\} ),\quad\textrm{where $n(x)\in \Int^d$}.$$

Assume there is $m \in \mathrm{int}(\mathcal W_H(x)_0) \cap \Int^d$ such that 
\begin{equation}\label{intersection-assumption}
T^m(x) \in T^{h_{n_k}}(U_{n_k})
\end{equation} 
for some $n_k$.

Since $m \in \mathrm{int}(\mathcal W_H(x)_0) \cap \Int^d$, one has that $$0 \in \mathrm{int}(-m+W_H(x, n(x))) = \mathrm{int}W_{H}(T^m(x), n(x) - m).$$ Hence $T^m(x) \in \iota_0(\mathcal W_H)$ and
\begin{equation}\label{pre-same-lower-level}
\mathcal W_{H}(T^m(x))_0 = W_H(T^m(x), n(x)-m).
\end{equation}

By the assumption \eqref{intersection-assumption}, there is $x_{n_k} \in U_{n_k}$ such that $$T^m(x) = T^{h_{n_k}}(x_{n_k}).$$ Then, with \eqref{contain-R-ball}, one has 
$$B_{R_1}(0) \subseteq W_H(T^{h_{n_k}}(x_{n_k}), n_k-h_{n_k}) = W_H(T^m(x), n_k-h_{n_k}),$$ and therefore (with \eqref{pre-same-lower-level}), 
$$W_{H}(T^m(x), n(x) - m) = W_H( T^{h_{n_k}}(x_{n_k}), n_k-h_{n_k})$$ and
$$V(T^m(x), n(x) - m) = V( T^{h_{n_k}}(x_{n_k}), n_k-h_{n_k}).$$
Hence, at the $-sH$ level, one also has
\begin{equation}\label{same-lower-level}
W_{sH}(T^m(x), n(x) - m) = W_{sH}( T^{h_{n_k}}(x_{n_k}), n_k-h_{n_k}) = -h_{n_k} + W_{sH}(x_{n_k}, n_k) .
\end{equation}

By \eqref{image-nbhd}, $h_{n_k}$ is in the $(2R_0+4+\sqrt{d}/2)$-neighbourhood of the $H$-projective image of $\partial W_{sH}(x_{n_k}, n_k)$, and therefore $0$ is in the $(2R_0+4+\sqrt{d}/2)$-neighbourhood of the $H$-projective image of $$-h_{n_k}+\partial W_{sH}(x_{n_k}, n_k) = W_{sH}( T^{h_{n_k}}(x_{n_k}), n_k-h_{n_k}) $$ Thus, by \eqref{same-lower-level},
the origin $0$ is in the $(2R_0+4+\sqrt{d}/2)$-neighbourhood of the $H$-projective image of $\partial W_{sH}(T^m(x), n(x)-m)$, and hence $m$ is in the $(2R_0+4+\sqrt{d}/2)$-neighbourhood of the $H$-projective image of $\partial W_{sH}(x, n(x))$, which is denoted by $\partial W_{sH}^H(x, n(x))$.

Therefore, for any $c\in \mathrm{int}(\mathcal W_H(x)_0) \cap \Int^d$ with $\mathrm{dist}(c, \partial\mathcal W_H)>N\sqrt{d}$, since $$c+n \in \mathrm{int}(\mathcal W_H(x)_0),\quad n \in \{0, 1, ..., N-1\}^d,$$ one has 
\begin{eqnarray*} 
& & \left\{ n \in \{0, 1, ..., N-1\}^d: T^{c+n}(x) \in \bigcup_{i=1}^K h_i(U_i) \right\} \\
& \subseteq & \left\{n \in \{0, 1, ..., N-1\}^d: \mathrm{dist}(c+n, \partial W_{sH}^H(x, n(x))) < 2R_0+4+\sqrt{d}/2 \right\}.
\end{eqnarray*}
Hence, by the choice of $N$ and Lemma \ref{convex-small-bd-N-Z},
\begin{eqnarray*}
& & \frac{1}{N^d} \abs{\left\{n \in \{0, 1, ..., N-1\}^d: T^{c+n}(x) \in \bigcup_{i=1}^K h_i(U_i)\right\}} \\
& \leq & \frac{1}{N^d} \abs{\left\{n\in c+ \{0, 1, ..., N-1\}^d : \mathrm{dist}(n, \partial W_{sH}^H(x, n(x))) < 2R_0+4+\sqrt{d}/2\right\}} \\
& < & \eps.
\end{eqnarray*}
This proves Property (\ref{insert-prop-3}).
\end{proof}

\section{Two towers}

\subsection{Rokhlin towers} 
Let $x \mapsto \mathcal W(x) = \bigcup_{n\in \Int^d} W(x, n)$ be a map with $\mathcal W(x)$ a tiling of $\Real^d$ and $W(x, n)$ is the cell with label $n$. Assume that the map $x\mapsto \mathcal W(x)$ is continuous in the sense that for any $\eps>0$ and any $W(x, n)$ with non-empty interior, if $y\in X$ is sufficiently close to $x$ then the Hausdorff distance between $W(x, n)$ and $W(y, n)$ are smaller than $\eps$. One also assumes that the map $x \mapsto \mathcal W(x)$  is equivariant in the sense that
$$W(T^{-m}(x), n+m) = m + W(x, n),\quad x\in X,\  m, n\in\Int^d.$$ 

The tiling functions $\mathcal W_H$ and $\mathcal W_{sH}$ constructed in the previous section clearly satisfy the assumptions above. With a such tiling function, one actually can build a Rokhlin tower as the following:

Let $N\in \mathbb N$ be arbitrary. Put
$$\Omega=\{x \in X: \mathrm{dist}(0, \partial \mathcal W(x)) > N\sqrt{d}\ \ \mathrm{and}\ \ \textrm{$\mathcal W(x)_0 = W(x, n)$ for some $n = 0 \mod N$}\},$$ where by $n = 0 \mod N$, one means $n_i = 0 \mod N$, $i=1, 2, ..., d$, if $n=(n_1, n_2, ..., n_d) \in \Int^d$.  
Note that $\Omega$ is open.

Let $m \in \{0, 1, ..., N-1\}^d$. Pick arbitrary $x \in \Omega$ and consider $T^{-m}(x)$. Note that $0 \in W(x, n)$ for some $n = 0 \mod N$ and $\mathrm{dist}(0, \partial W(x, n)) > N\sqrt{d}$. Since $$W(T^{-m}(x), n+m) = m + W(x, n),$$ one has that 
$$0 \in \mathrm{int}W(T^{-m}(x), n+m)\quad \textrm{and}\quad n+m = m \mod N.$$ 
Hence
\begin{equation}\label{upper-set}
T^{-m}(\Omega) \subseteq \Omega_m'.
\end{equation}
where
$$\Omega_m':=\{x \in X: 0 \notin\partial\mathcal W(x)\ \mathrm{and}\ \mathcal W(x)_0 = W(x, n), \ n = m \mod N\}.$$
For the same reason, if one defines 
$$\Omega''_m:=\{x \in X: \mathrm{dist}(0, \partial \mathcal W(x)) > 2N\sqrt{d}\ \mathrm{and}\ \mathcal W(x)_0 = W(x, n), \ n = m \mod N\},$$ 
then
\begin{equation}\label{lower-set}
 \Omega''_m \subseteq T^{-m}(\Omega).
 \end{equation}

Since the sets $$\Omega_m',\quad m \in \{0, 1, ..., N-1\}^d,$$ are mutually disjoint, it follows from \eqref{upper-set} that
$$ T^{-m}(\Omega),\quad m \in \{0, 1, ..., N-1\}^d$$
are mutually disjoint. That is, it forms a Rokhlin tower for $(X, T, \Int^d)$.

On the other hand, by \eqref{lower-set} and the construction of $\Omega_m''$, one has
\begin{equation}\label{tower-subset}
 \bigsqcup_{m \in \{0, 1, ..., N-1\}^d} T^{-m}(\Omega)  \supseteq  \bigsqcup_{m \in \{0, 1, ..., N-1\}^d} \Omega''_m  
 =\{x \in X: \mathrm{dist}(0, \partial\mathcal W(x)) > 2N\sqrt{d}\}.
\end{equation}
In particular, one has
\begin{equation}\label{ocap-control} 
\mathrm{ocap}\left(X \setminus  \bigsqcup_{m \in \{0, 1, ..., N-1\}^d} T^{-m}(\Omega)\right) \leq \mathrm{ocap}(\{x \in X: \mathrm{dist}(0, \partial\mathcal W(x)) \leq 2N\sqrt{d}\}).
\end{equation}

\begin{lem}\label{ocap-control-lem}
For any $E>0$, one has
$$\mathrm{ocap}(\{x \in X: \mathrm{dist}(0, \partial\mathcal W(x)) \leq E \}) \leq \limsup_{R \to \infty} \frac{1}{\mathrm{vol}(B_R)}\sup_{x\in X} \mathrm{vol}(\partial_E \mathcal W(x)\cap B_R),$$
where $\partial_E\mathcal W(x) = \{\xi\in\Real^d: \mathrm{dist}(\xi, \partial W(x)) \leq E\}$.
\end{lem}
\begin{proof}
Pick an arbitrary $x \in X$ and an arbitrary positive number $R$, and consider the partial orbit $$T^{m}(x),\quad \norm{m} < R.$$ Note that if $\mathrm{dist}(0, \partial\mathcal W(T^m(x))) \leq E $ (i.e., $0 \in \partial_E\mathcal W(T^m(x))$) for some $m$, then
$$-m \in \partial_E\mathcal W(x).$$ Therefore
$$\{\norm{m} < R: 0 \in \partial_E\mathcal W(T^m(x))\} \subseteq \{\norm{m} < R: m \in \partial_E\mathcal W(x)\}.$$
As $N \to \infty$, one has
\begin{eqnarray*}
&& \frac{1}{\abs{B_R \cap \Int^d}} \abs{\{\norm{m} < R: 0 \in \partial_E\mathcal W(T^m(x))\} } \\
& \leq & \frac{1}{\abs{B_R \cap \Int^d}} \abs{\{\norm{m} < R: m \in \partial_E\mathcal W(x)\} } \\
& \approx & \frac{1}{\mathrm{vol}(B_R)} \mathrm{vol}(\partial_E\mathcal W(x) \cap B_R),\quad \textrm{(if $R$ is sufficiently large).}
\end{eqnarray*}
Hence
$$\limsup_{R \to\infty} \frac{1}{\abs{B_R \cap \Int^d}} \abs{\{\abs{m} < R: 0 \in \partial_E\mathcal W(T^m(x))\} }  \leq \limsup_{R \to \infty} \frac{1}{\mathrm{vol}(B_R)}\sup_{x\in X} \mathrm{vol}(\partial_E \mathcal W(x)\cap B_R).$$
Since $x$ is arbitrary, this proves the desired conclusion.
\end{proof}

\begin{thm}\label{URP}
Consider the minimal free dynamical system $(X, T, \Int^d)$. Then, for any $\eps>0$ and $N \in\mathbb N$, there is an open set $\Omega\subseteq X$ such that
$$T^{-n}(\Omega),\quad n\in\{0, 1, ..., N-1\}^d$$
are mutually disjoint (hence form a Rokhlin tower), and
$$\mathrm{ocap}\left(X\setminus\bigcup_{n\in \{0, 1, ..., N-1\}^d}T^{-n}(\Omega)\right) < \eps.$$ In other words, the system $(X, T, \Int^d)$ has the Uniform Rohklin Property in the sense of Definition 7.1 of \cite{Niu-MD-Z} (see Lemma 7.2 of \cite{Niu-MD-Z}).
\end{thm}
\begin{proof}
By Lemma 4.2 of \cite{GLT-Zk}, there is an equivariant $\Real^d$-tiling $x\mapsto \mathcal W(x)$ such that 
$$\limsup_{R \to \infty} \frac{1}{\mathrm{vol}(B_R)}\sup_{x\in X} \mathrm{vol}(\partial_{2N\sqrt{d}} \mathcal W(x)\cap B_R) < \eps.$$ Then, the statement follows from \eqref{ocap-control} and Lemma \ref{ocap-control-lem} (with $E=2N\sqrt{d}$).
\end{proof}

\subsection{The two towers}

The Rokhlin tower constructed above in general cannot cover the whole space $X$. Consider the two continuous tiling functions $\mathcal W_{sH}$ and $\mathcal W_{H}$, and consider the Rokhlin towers $\mathcal T_0$ and $\mathcal T_1$ constructed from them respectively. It is still possible that $\mathcal T_0$ together with $\mathcal T_1$ do not cover the whole space $X$. However, in the following theorem, one can show that the complement of the tower $\mathcal T_0$ can be cut into pieces and then each piece can be translated into the tower $\mathcal T_1$ in a way that the order of the overlaps of the translations are universally bounded, and the intersection of the translations with each $\mathcal T_1$-orbit is uniformly small. This eventually leads to a  Cuntz comparison of open sets for minimal free $\Int^d$-actions (Theorem \ref{open-set-comparison}).

\begin{thm}\label{two-towers}
Consider a minimal free dynamical system $(X, T, \Int^d)$. Let $N\in\mathbb N$ and $\eps>0$ be arbitrary. There exist two Rokhlin towers $$\mathcal T_{0}:=\{T^{-m}(\Omega_{0}): m \in \{0, 1, ..., N_0-1\}^d\}\quad \textrm{and}\quad \mathcal T_{1}:=\{T^{-m}(\Omega_{1}): m \in \{0, 1, ..., N_1-1\}^d\},$$ with $N_0, N_1 \geq N$ and $\Omega_0, \Omega_1\subseteq X$ open, an open cover $\{U_1, U_2, ..., U_K\}$ of  $X \setminus \bigcup_{m} T^{-m}(\Omega_{0})$, and $h_1, h_2, ..., h_K \in\Int^d$ such that
\begin{enumerate}
\item $T^{h_k}(U_k) \subseteq  \bigcup_{m} T^{-m}(\Omega_{1})$, $k=1, 2, ..., K$;
\item the open sets $$T^{h_k}(U_k),\quad k=1, 2, ..., K,$$ can be grouped as 
$$
\left\{
\begin{array}{l}
T^{h_{1}}(U_{1}), ..., T^{h_{s_1}}(U_{s_1}), \\
T^{h_{s_1+1}}(U_{s_1+1}), ..., T^{h_{s_2}}(U_{s_2}), \\
\cdots \\
T^{h_{s_{m-1}+1}}(U_{s_{m-1}+1}), ..., T^{h_{s_m}}(U_{s_m}),
\end{array}
\right.
$$
for some $m \leq (\lfloor2\sqrt{d}\rfloor + 1)^d$, such that the open sets in each group are mutually disjoint; 
\item for each $x\in \Omega_{1}$, one has
         $$\frac{1}{N_1^d} \abs{\left\{m \in \{0, 1, ..., N_1-1\}^d: T^m(x) \in \bigcup_{k=1}^K T^{n_k}(U_{k})\right\}} < \eps.$$
\end{enumerate}
\end{thm}
\begin{proof}
Applying Lemma \ref{pre-two-towers} with $R_0 = 2N\sqrt{d}$, $\eps$, and some $s\in (1, 2)$, together with some $N_1>\max\{N(R_0, \eps), N\}$ (in place of $N$) and $R_1 > \max\{R_0, 2N_1\sqrt{d}\}$, where $N(R_0, \eps)$ is the constant of Lemma \ref{convex-small-bd-N-Z} with respect to $\eps$ and $2R_0+4+\sqrt{d}/2$,  there are two continuous equivariant $\Real^d$-tilings $\mathcal W_{sH}$ and $\mathcal W_H$ for some (sufficiently large) $H>0$, a finite open cover
$$U_1 \cup U_2 \cup \cdots \cup U_K \supseteq \beta_{R_0}(\mathcal W_{sH}),$$ and $n_1, n_2, ..., n_K \in \Int^d$ such that
\begin{enumerate}
\item\label{pre-prop-1} $T^{n_i}(U_i) \subseteq \iota_{R_1}(\mathcal W_H) \subseteq \iota_0(\mathcal W_H)$, $i=1, 2, ..., K$;
\item\label{pret-prop-2} the open sets $$T^{n_i}(U_i),\quad i=1, 2, ..., K,$$ can be grouped as 
$$
\left\{
\begin{array}{l}
T^{n_{1}}(U_{1}), ..., T^{n_{s_1}}(U_{s_1}), \\
T^{n_{s_1+1}}(U_{s_1+1}), ..., T^{n_{s_2}}(U_{s_2}), \\
\cdots \\
T^{n_{s_{m-1}+1}}(U_{s_{m-1}+1}), ..., T^{n_{s_m}}(U_{s_m}),
\end{array}
\right.
$$
with $m \leq (\lfloor2\sqrt{d}\rfloor + 1)^d$, such that the open sets in each group are mutually disjoint; 
\item\label{pre-prop-3} for each $x\in \iota_0(\mathcal W_H)$ and each  $c\in \mathrm{int}(\mathcal W_H(x)_0) \cap \Int^d$ with $\mathrm{dist}(c, \partial\mathcal W_H)>N_1\sqrt{d}$, one has
         $$\frac{1}{N_1^d} \abs{\left\{n \in \{0, 1, ..., N_1-1\}^d: T^{c+n}(x) \in \bigcup_{i=1}^K T^{n_i}(U_i)\right\}} < \eps.$$  
\end{enumerate}

Put
$$\Omega_{0}=\{x \in X: \mathrm{dist}(0, \partial \mathcal W_{sH}(x))> N\sqrt{d} \ \mathrm{and}\ \mathcal W_{sH}(x)_0 = W_{sH}(x, n), \ n = 0 \mod N\}.$$
Then $$T^{-m}(\Omega_0),\quad m\in\{0, 1, ..., N_0-1\}^d$$ form a Rokhlin tower with $N_0=N$, and by \eqref{tower-subset}
\begin{equation}
X\setminus \bigsqcup_{m \in \{0, 1, ..., N_0-1\}^d} T^{-m}(\Omega_0)  \subseteq  \{x \in X: \mathrm{dist}(0, \partial \mathcal W_{sH}(x)) \leq 2N\sqrt{d}\} = \beta_{2N\sqrt{d}}(\mathcal W_{sH}).
\end{equation}
Thus, $U_1, U_2, ..., U_K$ form an open cover of $X\setminus \bigsqcup_{m \in \{0, 1, ..., N_0-1\}^d} T^{-m}(\Omega_0)$.

Put $$\Omega_{1}=\{x \in X: \mathrm{dist}(0, \partial \mathcal W_H(x)) > N_1\sqrt{d} \ \mathrm{and}\ \mathcal W_{H}(x)_0 = W_{H}(x, n), \ n = 0 \mod N_1\}.$$
Then $$T^{-m}(\Omega_1),\quad m\in\{0, 1, ..., N_1-1\}^d,$$ form a Rokhlin tower, and by \eqref{tower-subset} (and the assumption that $R_1 > 2N_1\sqrt{d}$), 
\begin{equation}
 \bigsqcup_{m \in \{0, 1, ..., N_1-1\}^d} T^{-m}(\Omega_1)  \supseteq     \{x \in X: \mathrm{dist}(0, \partial \mathcal W_H) > 2N_1\sqrt{d}\} \supseteq \iota_{R_1}(\mathcal W_H).
\end{equation}
Thus, $T^{-h_i}(U_i) \subseteq \bigsqcup_{m \in \{0, 1, ..., N_1-1\}^d} T^{-m}(\Omega_1)$.

If $x\in \Omega_1$ (hence $x\in\iota_0(\mathcal W_H)$ and $\mathrm{dist}(0, \partial\mathcal W_H)>N_1\sqrt{d}$ ), it then follows from \eqref{pre-prop-3} (with $c=0$) that $$\frac{1}{N_1^d} \abs{\left\{m \in \{0, 1, ..., N_1-1\}^d: T^m(x) \in \bigcup_{k=1}^K T^{n_k}(U_{k})\right\}} < \eps,$$
as desired.
\end{proof}

\section{Cuntz comparison of open sets, comparison radius, and the mean topological dimension}

With the two-tower construction in the previous section, one is able to show that the C*-algebra $\mathrm{C}(X)\rtimes\Int^d$ has Cuntz-comparison on open sets (Theorem \ref{open-set-comparison}), and therefore the radius of comparison of $\mathrm{C}(X)\rtimes\Int^d$ is at most half of the mean dimension of $(X, T, \Int^d)$.

As a preparation, one has the following two very simple observations on the Cuntz semigroup of a C*-algebra.
\begin{lem}
Let $A$ be a C*-algebra, and let $a_1, a_2, ..., a_m \in A$ be positive elements. Then
$$[a_1] + [a_2] + \cdots + [a_m] \leq m[a_1+a_2+\cdots+a_m].$$
\end{lem}
\begin{proof}
The lemma follows from the observation:
$$
\left(
\begin{array}{cccc}
a_1 & & &  \\
 & a_2 & & \\
& & \ddots & \\
& & & a_m  
\end{array}
\right) 
\leq 
\left(
\begin{array}{cccc}
a_1+\cdots+a_m & & &  \\
 & a_1+\cdots+a_m & & \\
& & \ddots & \\
& & & a_1+\cdots + a_m  
\end{array}
\right).
$$
\end{proof}

\begin{lem}\label{Cuntz-overlap}
Let $U_1, U_2, ..., U_K \subseteq X$ be open sets which can be divided into $M$ groups such that each group consists of mutually disjoint sets. Then
$$[\varphi_{U_1}] + \cdots + [\varphi_{U_K}] \leq M [\varphi_{U_1\cup\cdots\cup U_K}] = M [\varphi_{U_1} + \cdots + \varphi_{U_K}]$$
\end{lem}
\begin{proof}
Write $U_1, U_2, ..., U_K$ as 
$$\{U_{1}, ..., U_{s_1}\}, \{U_{s_1+1}, ..., U_{s_2}\}, ..., \{U_{s_{m-1}+1}, ..., U_{s_M}\},$$
such that the open sets in each group are mutually disjoint. Then
$$[\varphi_{U_{s_i+1}}] + \cdots + [\varphi_{U_{s_{i+1}}}] = [\varphi_{U_{s_i+1}} + \cdots + \varphi_{U_{s_{i+1}}}] = [\varphi_{U_{s_i+1}\cup\cdots\cup U_{s_{i+1}} }],\quad i=0, 1, ..., M-1,$$
and together with the lemma above, one has
\begin{eqnarray*}
[\varphi_{U_1}] + \cdots + [\varphi_{U_K}] & = & [\varphi_{U_{1}} + \cdots + \varphi_{U_{s_1}}] + \cdots + [\varphi_{U_{s_{m-1}+1}} + \cdots + \varphi_{U_{s_M}}] \\
&= & [\varphi_{U_{1} \cup\cdots\cup U_{s_1}}] + \cdots + [\varphi_{U_{s_{m-1}+1} \cup \cdots \cup U_{s_M}}] \\
& \leq & M [\varphi_{U_1\cup\cdots\cup U_K}],
\end{eqnarray*}
as desired.
\end{proof}

\begin{defn}
Consider a topological dynamical system $(X, \Gamma)$, where $X$ is a compact metrizable space and $\Gamma$ is a discrete group acting on $X$ from the right, and consider a Rokhlin tower
$$\mathcal T=\{\Omega\gamma,\ \gamma \in \Gamma_0\},$$
where $\Omega\subseteq X$ is open and $\Gamma_0\subseteq \Gamma$ is a finite set containing the unit $e$ of the discrete group $\Gamma$. Define the C*-algebra
$$\textrm{C*}(\mathcal T):=\textrm{C*}\{u_\gamma \mathrm{C}_0(\Omega),\ \gamma\in\Gamma_0\}\subseteq \mathrm{C}(X) \rtimes \Gamma.$$ By Lemma 3.11 of \cite{Niu-MD-Z}, it is canonically isomorphic to $\mathrm{M}_{\abs{\Gamma_0}}(\mathrm{C}_0(\Omega))$, and $$\mathrm{C}_0(\bigcup_{\gamma\in\Gamma_0}\Omega\gamma ) \ni \phi \mapsto \mathrm{diag}\{\phi|_{\Omega\gamma_1}, \phi|_{\Omega\gamma_2}, ..., \phi|_{\Omega\gamma_{\abs{\Gamma_0}}}\} \in \mathrm{M}_{\abs{\Gamma_0}}(\mathrm{C}_0(\Omega))\}$$ under this isomorphism.
\end{defn}

The following comparison result essentially is a special case of Theorem 7.8 of \cite{Niu-MD-Z}.
\begin{lem}[Theorem 7.8 of \cite{Niu-MD-Z}]\label{diagonal-comparison}
Let $Z$ be a locally compact metrizable space, and consider $\mathrm{M}_n(\mathrm{C}_0(Z))$. Let $a, b \in \mathrm{M}_n(\mathrm{C}_0(Z))$ be two positive diagonal elements, i.e.,
$$a(t)=\mathrm{diag}\{a_1(t), a_2(t), ..., a_n(t)\} \quad\mathrm{and}\quad b(t)=\mathrm{diag}\{b_1(t), b_2(t), ..., b_n(t)\}$$
for some positive continuous functions $a_1, ..., a_n, b_1, ..., b_n: Z \to \Real$. If $$\mathrm{rank}(a(t)) \leq \frac{1}{4}\mathrm{rank}(b(t)),\quad t\in Z$$ and
$$4< \mathrm{rank}(b(t)),\quad t\in Z,$$
then $a \precsim b$ in $\mathrm{M}_n(\mathrm{C}_0(Z))$.
\end{lem}

\begin{proof}
It is enough to show that $(a-\eps)_+ \precsim b$ for arbitrary $\eps>0$. For a given $\eps>0$, there is a compact subset $D \subseteq Z$ such that $(a-\eps)_+$ is supported inside $D$. Denote by $\pi: \mathrm{M}_n(\mathrm{C}_0(Z)) \to \mathrm{M}_{n}(\mathrm{C}(D))$ the restriction map. One then has 
$$\mathrm{rank}(\pi((a-\eps)_+)(t)) \leq \frac{1}{4}\mathrm{rank}(\pi(b)(t)), \quad t\in D$$ and 
$$\frac{1}{n} < \frac{1}{4n}\mathrm{rank}(b(t)),\quad t\in D.$$
By Theorem 7.8 of \cite{Niu-MD-Z}, one has that $\pi((a-\eps)_+) \precsim \pi(b)$ in $\mathrm{M}_n(\mathrm{C}(D))$, that is, there is a sequence $(v_k) \subseteq \mathrm{M}_n(\mathrm{C}(D))$ such that $v_k(\pi(b))v_k^* \to \pi((a-\eps)_+)$ as $k \to \infty$. Extend each $v_k$ to a function in $\mathrm{M}_n(\mathrm{C}_0(Z))$, and still denote it by $v_k$. It is clear that the new sequence $(v_k)$ satisfies $v_kbv_k^* \to (a-\eps)_+$ as $k\to\infty$, and hence $(a - \eps)_+ \precsim b$, as desired.
\end{proof}

\begin{thm}\label{open-set-comparison}
Let $(X, T, \Int^d)$ be a minimal free dynamical system, and let $E, F\subseteq X$ be open sets such that
$$\mu(E) \leq \frac{1}{4}\nu(F),\quad \mu\in\mathcal M_1(X, T, \Int^d).$$ Then,
$$[\varphi_E] \leq ((2\lfloor\sqrt{d}\rfloor + 1)^d + 1)[\varphi_F]$$
in the Cuntz semigroup of $\mathrm{C}(X) \rtimes \Int^d$. In other words, the C*-algebra $\mathrm{C}(X) \rtimes \Int^d$ has $(\frac{1}{4}, (2\lfloor\sqrt{d}\rfloor + 1)^d + 1)$-Cuntz-comparison on open sets in the sense of Definition 4.1 of \cite{Niu-MD-Z}.
\end{thm}
\begin{proof}
Let $E$ and $F$ be open sets satisfying the condition of the theorem. Let $\eps>0$ be arbitrary. In order to prove the statement of the theorem, it is enough to show that $$(\varphi_E - \eps)_+ \precsim \underbrace{\varphi_F \oplus\cdots\oplus \varphi_F}_{(2\lfloor\sqrt{d}\rfloor + 1)^d + 1}.$$ 

For the given $\eps$, pick a compact set $E'\subseteq E$ such that 
\begin{equation}\label{cpt-E}
(\varphi_E-\eps)_+(x) = 0,\quad x\notin E'.
\end{equation} 
By the assumption of the theorem, one has that 
\begin{equation}\label{mea-small}
\mu(E') < \frac{1}{4} \mu(F),\quad \mu\in\mathcal M_1(X, T, \Int^n),
\end{equation}
and then there is $N\in\mathbb N$ such that for any $M > N$ and any $x\in X$, 
\begin{equation}\label{small-frq}
\frac{1}{M^d}\{m \in \{0, 1, ..., M-1\}^d: T^{-m}(x) \in E'\} < \frac{1}{4} \frac{1}{M^d}\{m \in \{0, 1, ..., M-1\}^d: T^{-m}(x) \in F\}.\end{equation}
Otherwise, there are sequences $N_k\in\mathbb N$, $x_k\in X$, $k=1, 2, ..., $ such that $N_k\to\infty$ as $k\to\infty$, and for any $k$, 
$$\frac{1}{N_k^d}\{m \in \{0, 1, ..., N_k-1\}^d: T^{-m}(x_k) \in E'\} \geq \frac{1}{4} \frac{1}{N_k^d}\{m \in \{0, 1, ..., N_k-1\}^d: T^{-m}(x_k) \in F\}.$$ That is
\begin{equation}\label{mea-contradic}
4\delta_{N_k, x_k}(E') \geq \delta_{N_k, x_k}(F),\quad k=1, 2, ... ,
\end{equation}
where $\delta_{N_k, x_k}=\frac{1}{N_k^d}\sum_{m\in\{0, 1, ..., N_k-1\}^d}\delta_{T^{-m}(x_k)}$ and $\delta_y$ is the Diract measure concentrated at $y$. 
Let $\delta_\infty$ be a limit point of $\{\delta_{N_k, x_k}, k=1, 2, ...\}$ and it is clear that $\delta_\infty\in \mathcal M_1(X, T, \Int^d)$.
Passing to a subsequence of $k$, one has 
\begin{eqnarray*}
\delta_\infty(F) &\leq & \liminf_{k\to\infty} \delta_{N_k, x_k}(F)\quad\quad\quad \textrm{($F$ is open)} \\
&  \leq & 4 \liminf_{k\to\infty} \delta_{N_k, x_k}(E')\quad\quad\quad \textrm{(by \eqref{mea-contradic})}\\
& \leq & 4 \limsup_{k\to\infty} \delta_{N_k, x_k}(E') \\
& \leq & 4 \delta_\infty(E'), \quad\quad\quad \textrm{($E'$ is closed)}
\end{eqnarray*}
which contradicts to \eqref{mea-small}.

With \eqref{cpt-E} and \eqref{small-frq}, one has that for any $M > N$ and any $x\in X$, 
\begin{eqnarray}\label{rank-ineq}
& & \frac{1}{M^d}\{m\in \{0, 1, ..., M-1\}^d: (\varphi_E-\eps)_+(T^{-m}(x)) > 0\} \\
&\leq & \frac{1}{M^d}\{m \in \{0, 1, ..., M-1\}^d: T^{-m}(x) \in E'\}\nonumber \\
&< & \frac{1}{4} \frac{1}{M^d}\{m \in \{0, 1, ..., M-1\}^d: T^{-m}(x) \in F\} \nonumber \\
& = & \frac{1}{4}\frac{1}{M^d}\{m\in \{0, 1, ..., M-1\}^d: \varphi_F(T^{-m}(x)) > 0\}. \nonumber
\end{eqnarray}
Also note that since $(X, T, \Int^d)$ is minimal, there is $\delta>0$ such that for any $M>N$, 
\begin{equation}\label{lbd-density}
\frac{1}{4M^d}\abs{\{ m \in \{0, 1, ..., M-1\}^d: \varphi_F(T^{-m}(x)) > 0\}} > \delta,\quad x\in X
\end{equation}

Let $$\mathcal T_0=\{T^{-m}(\Omega_0),\quad m\in \{0, 1, ..., N_0-1\}^d\}$$ and
$$\mathcal T_1=\{T^{-m}(\Omega_1),\quad m\in \{0, 1, ..., N_1-1\}^d\}$$
be the two towers obtained from Theorem \ref{two-towers} with respect to $\max\{N, \sqrt[d]{\frac{1}{\delta}}\}$ and $\delta$. Denote by $U_1, U_2, ..., U_K$ and $n_1, n_2, ..., n_K\in \Int^d$ be the open sets and group elements, respectively, obtained from Theorem \ref{two-towers}.

Pick $\chi_0 \in \mathrm{C}(X)^+$ such that 
\begin{equation}\label{cut-function}
\left\{ 
\begin{array}{ll}
\chi_0(x) = 1, & x\notin \bigcup_{k=1}^K U_k, \\
\chi_0(x) > 0, & x \in \bigsqcup_{m\in \{0, 1, ..., N_0-1\}^d} T^{-m}(\Omega_0), \\
\chi_0(x) = 0, & x \notin \bigsqcup_{m\in \{0, 1, ..., N_0-1\}^d} T^{-m}(\Omega_0).
\end{array}
\right.
\end{equation}
Note that then $(1-\chi_0)$ is supported in $U_1\cup U_2\cup\cdots\cup U_K$. Consider
$$(\varphi_{E} - \eps)_+ = (\varphi_{E} - \eps)_+(1-\chi_0) + (\varphi_{E} - \eps)_+\chi_0.$$
Then, for any $x\in \Omega_0$, it follows from \eqref{rank-ineq} and \eqref{cut-function} that
\begin{eqnarray*}
& & \abs{\{m\in\{0, 1, ..., N_0-1\}^d: ((\varphi_{E} - \eps)_+\chi_0)(T^{-m}(x))>0 \} }\\
& = & \abs{\{m\in\{0, 1, ..., N_0-1\}^d: (\varphi_{E} - \eps)_+(T^{-m}(x))>0 \}} \\
& < & \frac{1}{4} \abs{\{m\in\{0, 1, ..., N_0-1\}^d: \varphi_{F}(T^{-m}(x))>0 \} } \\
& = & \frac{1}{4} \abs{\{m\in\{0, 1, ..., N_0-1\}^d: (\varphi_{F}\chi_0)(T^{-m}(x))>0 \}}.
\end{eqnarray*}
Therefore, under the isomorphism $\textrm{C*}(\mathcal T_0) \cong \mathrm{M}_{N^d_0}(\mathrm C_0(\Omega_0))$, one has 
$$\mathrm{rank}(((\varphi_E-\eps)_+\chi_0)(x)) \leq \frac{1}{4}\mathrm{rank}((\varphi_F\chi_0)(x)),\quad x\in \Omega_0.$$
Moreover, it follows from \eqref{lbd-density} and the fact that $N_0> \sqrt[d]{\frac{1}{\delta}}$ that for any $x\in\Omega_0$, 
$$\frac{1}{4N_0^d}\mathrm{rank}((\varphi_F\chi_0)(x)) = \frac{1}{4N_0^d}\abs{\{m\in\{0, 1, ..., N_0-1\}^d: \varphi_F(T^{-m}(x)) > 0  \}}  > \delta > \frac{1}{N^d_0}.$$
Thus, by Lemma \ref{diagonal-comparison}, one has that 
\begin{equation}\label{first-tower}
(\varphi_E-\eps)_+\chi_0 \precsim \varphi_F\chi_0 \precsim \varphi_F.
\end{equation}

Consider $(\varphi_E-\eps)_+(1-\chi_0)$. Since $(1-\chi_0)$ is supported in $U_1 \cup U_2 \cup\cdots\cup U_K$, one has that $$(\varphi_E-\eps)_+(1-\chi_0) \precsim (1-\chi_0)  \precsim \varphi_{U_1 \cup\cdots\cup U_K} \sim \varphi_{U_1} + \cdots + \varphi_{U_K}\precsim \varphi_{U_1} \oplus \cdots \oplus \varphi_{U_K}.$$
On the other hand, by Lemma \ref{Cuntz-overlap}, 
$$\varphi_{T^{n_1}(U_1)} \oplus \cdots \oplus \varphi_{T^{n_K}(U_K)} \precsim \bigoplus_{(2\lfloor\sqrt{d}\rfloor + 1)^d} (\varphi_{T^{n_1}(U_1)} + \cdots + \varphi_{T^{n_K}(U_K)}).$$
Note that $\varphi_{U_i} \sim \varphi_{T^{n_i}(U_i)}$, $i=1, 2, ..., K$, and one has
\begin{equation}\label{second-tower}
(\varphi_E-\eps)_+(1-\chi_0) \precsim \bigoplus_{(2\lfloor\sqrt{d}\rfloor + 1)^d} (\varphi_{T^{n_1}(U_1)\cup\cdots\cup T^{n_K}(U_K)}).
\end{equation}
By Theorem \ref{two-towers},
\begin{equation}\label{density-small-sec}
\frac{1}{N_1^d}\abs{\left\{m \in\{0, 1, ..., N_1-1\}^d: T^{-m}(x) \in \bigcup_{k=1}^K T^{n_k}(U_k) \right\}} < \delta,\quad x\in\Omega_1.
\end{equation}
Let $\chi_1: X \to [0, 1]$ be a continuous function such that 
\begin{equation*}
\left\{ 
\begin{array}{ll}
\chi_1(x) > 0, & x \in \bigsqcup_{m\in \{0, 1, ..., N_1-1\}^d} T^{-m}(\Omega_1), \\
\chi_1(x) = 0, & x \notin \bigsqcup_{m\in \{0, 1, ..., N_1-1\}^d} T^{-m}(\Omega_1).
\end{array}
\right.
\end{equation*}
Then $$\frac{1}{4N_1^d}\mathrm{rank}((\varphi_F\chi_1)(x)) = \frac{1}{4N_1^d}\abs{\{m\in\{0, 1, ..., N_1-1\}^d: \varphi_F(T^{-m}(x)) > 0  \}} > \delta > \frac{1}{N^d_1},\quad x\in\Omega_1,$$
and hence, for any $x\in\Omega_1$, with \eqref{density-small-sec}, one has
\begin{eqnarray*}
\mathrm{rank}(\varphi_{T^{n_1}(U_1)\cup\cdots\cup T^{n_K}(U_K)}(x)) 
&=& \abs{\left\{m \in\{0, 1, ..., N_1-1\}^d: T^{-m}(x) \in \bigcup_{k=1}^K T^{n_k}(U_k) \right\}} \\
& < & N_1^d\delta  <  \frac{1}{4}\mathrm{rank}((\varphi_F\chi_1)(x)). 
\end{eqnarray*}
By Lemma \ref{diagonal-comparison},
$$\varphi_{T^{n_1}(U_1)\cup\cdots\cup T^{n_K}(U_K)} \precsim  \varphi_F\chi_1 \precsim\varphi_F,$$ and together with \eqref{second-tower} and \eqref{first-tower},
\begin{eqnarray*}
(\varphi_{E} - \eps)_+ & \precsim & (\varphi_{E} - \eps)_+(1- \chi_0) \oplus (\varphi_{E} - \eps)_+ \chi_0 \\
&\precsim & (\bigoplus_{(2\lfloor\sqrt{d}\rfloor + 1)^d} (\varphi_{T^{n_1}(U_1)\cup\cdots\cup T^{n_K}(U_K)})) \oplus \varphi_F \\
& \precsim & (\bigoplus_{(2\lfloor\sqrt{d}\rfloor + 1)^d} \varphi_F) \oplus \varphi_F,
\end{eqnarray*}
as desired.
\end{proof}

\begin{thm}\label{rc-mdim}
Let $(X, T, \Int^d)$ be a minimal free dynamical system. Then
$$\mathrm{rc}(\mathrm{C}(X) \rtimes \Int^d) \leq \frac{1}{2} \mathrm{mdim}(X, T, \Int^d).$$
\end{thm}
\begin{proof}
By Theorem \ref{open-set-comparison}, the C*-algebra $\mathrm{C}(X) \rtimes \Int^d$ has $(\frac{1}{4}, (2\lfloor\sqrt{d}\rfloor + 1)^d + 1)$-Cuntz-comparison on open sets. By Theorem \ref{URP}, the dynamical system $(X, T, \Int^d)$ has the Uniform Rohklin Property. Then the statement follows directly from Theorem 4.7 of \cite{Niu-MD-Z}.
\end{proof}

The following corollary generalizes Corollary 4.9 of \cite{EN-MD0} (where $d=1$) and generalizes the classifiability result of \cite{Szabo-Z} (where $\mathrm{dim}(X) < \infty$).
\begin{cor}\label{clasn}
Let $(X, T, \Int^d)$ be a minimal free dynamical system with mean dimension zero, then $\mathrm{C}(X) \rtimes \Int^d$ is classifiable. In particular, if $\mathrm{dim}(X) < \infty$, or $(X, T, \Int^d)$ has at most countably many ergodic measures, or $(X, T, \Int^d)$ has finite topological entropy, then $\mathrm{C}(X) \rtimes \Int^d$ is classifiable.
\end{cor}
\begin{proof}
If $(X, T, \Int^d)$ has mean dimension zero, then $\mathrm{rc}(\mathrm{C}(X) \rtimes \Int^d)=0$ by Theorem \ref{rc-mdim}; that is,  $\mathrm{C}(X) \rtimes \Int^d$ has strict comparison of positive elements. Note that, by Corollary 5.4 of \cite{GLT-Zk}, the dynamical system $(X, T, \Int^d)$ has small boundary property. Then by Corollary 9.5 of \cite{KS-comparison}, the C*-algebra  $\mathrm{C}(X) \rtimes \Int^d$ has finite nuclear dimension, and hence it is classifiable by \cite{GLN-TAS}, \cite{EN-K0-Z}, \cite{EGLN-DR}, and \cite{TWW-QD}.
\end{proof}

The following is a generalization of Corollary 5.7 of \cite{EN-MD0}.
\begin{cor}
Let $(X_1, T_1, \Int^{d_1})$ and $(X_2, T_2, \Int^{d_2})$ be minimal free dynamical systems where $d_1, d_2 \in\mathbb N$. Then the tensor product C*-algebra 
$(\mathrm{C}(X_1) \rtimes \Int^{d_1}) \otimes (\mathrm{C}(X_2) \rtimes \Int^{d_2})$ is classifiable.
\end{cor}
\begin{proof}
Note that $$(\mathrm{C}(X_1) \rtimes \Int^{d_1}) \otimes (\mathrm{C}(X_2) \rtimes \Int^{d_2}) \cong \mathrm{C}(X_1 \times X_2) \rtimes (\Int^{d_1} \times \Int^{d_2}),$$ where $\Int^{d_1} \times \Int^{d_2}$ acting on $X_1 \times X_2$ by $$(T_1 \times T_2)^{(n_1, n_2)}((x_1, x_2)) = (T_1^{n_1}(x_1), T_2^{n_2}(x_2)),\quad n_1\in \Int^{d_1},\  n_2\in\Int^{d_2}.$$ By the argument of Remark 5.8 of \cite{EN-MD0}, one has $$\mathrm{mdim}(X_1 \times X_2, T_1 \times T_2, \Int^{d_1} \times \Int^{d_2}) = 0,$$ and the statement then follows from Corollary \ref{clasn}.
\end{proof}

\bibliographystyle{plainurl}
\bibliography{operator_algebras}

\begin{thebibliography}{10}

\bibitem{EGLN-DR}
G.~A. Elliott, G.~Gong, H.~Lin, and Z.~Niu.
\newblock On the classification of simple amenable {C*}-algebras with finite
  decomposition rank, {II}.
\newblock 07 2015.
\newblock URL: \url{http://arxiv.org/abs/1507.03437}, \href
  {http://arxiv.org/abs/1507.03437} {\path{arXiv:1507.03437}}.

\bibitem{EN-RC}
G.~A. Elliott and Z.~Niu.
\newblock On the radius of comparison of a commutative {C*}-algebra.
\newblock {\em Canad. Math. Bull.}, 56(4):737--744, 2013.
\newblock URL: \url{https://doi.org/10.4153/CMB-2012-012-9}, \href
  {http://dx.doi.org/10.4153/CMB-2012-012-9}
  {\path{doi:10.4153/CMB-2012-012-9}}.

\bibitem{EN-K0-Z}
G.~A. Elliott and Z.~Niu.
\newblock On the classification of simple amenable {C*}-algebras with finite
  decomposition rank.
\newblock In R.~S. Doran and E.~Park, editors, {\em ``Operator Algebras and
  their Applications: A Tribute to Richard V.~Kadison", Contemporary
  Mathematics}, volume 671, pages 117--125. Amer. Math. Soc., 2016.
\newblock \href {http://arxiv.org/abs/http://dx.dot.org/10.1090/conm/671/13506}
  {\path{arXiv:http://dx.dot.org/10.1090/conm/671/13506}}.

\bibitem{EN-MD0}
G.~A. Elliott and Z.~Niu.
\newblock The {C{$^*$}}-algebra of a minimal homeomorphism of zero mean
  dimension.
\newblock {\em Duke Math. J.}, 166(18):3569--3594, 2017.
\newblock URL: \url{https://doi.org/10.1215/00127094-2017-0033}, \href
  {http://dx.doi.org/10.1215/00127094-2017-0033}
  {\path{doi:10.1215/00127094-2017-0033}}.

\bibitem{GLN-TAS}
G.~Gong, H.~Lin, and Z.~Niu.
\newblock Classification of finite simple amenable {$\mathcal Z$}-stable
  {C*}-algebras.
\newblock 01 2015.
\newblock URL: \url{http://arxiv.org/abs/1501.00135}, \href
  {http://arxiv.org/abs/1501.00135} {\path{arXiv:1501.00135}}.

\bibitem{Gromov-MD}
M.~Gromov.
\newblock Topological invariants of dynamical systems and spaces of holomorphic
  maps. {I}.
\newblock {\em Math. Phys. Anal. Geom.}, 2(4):323--415, 1999.
\newblock URL: \url{https://mathscinet.ams.org/mathscinet-getitem?mr=1742309},
  \href {http://dx.doi.org/10.1023/A:1009841100168}
  {\path{doi:10.1023/A:1009841100168}}.

\bibitem{GLT-Zk}
Yonatan Gutman, Elon Lindenstrauss, and Masaki Tsukamoto.
\newblock Mean dimension of {$\Bbb{Z}^k$}-actions.
\newblock {\em Geom. Funct. Anal.}, 26(3):778--817, 2016.
\newblock URL: \url{http://dx.doi.org/10.1007/s00039-016-0372-9}, \href
  {http://dx.doi.org/10.1007/s00039-016-0372-9}
  {\path{doi:10.1007/s00039-016-0372-9}}.

\bibitem{Haagtrace}
U.~Haagerup.
\newblock Quasitraces on exact {$\textrm C^*$}-algebras are traces.
\newblock {\em C. R. Math. Acad. Sci. Soc. R. Can.}, 36(2-3):67--92, 2014.

\bibitem{KS-comparison}
D.~Kerr and G.~Szabo.
\newblock Almost finiteness and the small boundary property.
\newblock 07 2018.
\newblock URL: \url{https://arxiv.org/pdf/1807.04326}, \href
  {http://arxiv.org/abs/1807.04326} {\path{arXiv:1807.04326}}.

\bibitem{Lind-MD}
E.~Lindenstrauss.
\newblock Mean dimension, small entropy factors and an embedding theorem.
\newblock {\em Inst. Hautes {\'E}tudes Sci. Publ. Math.}, (89):227--262 (2000),
  1999.
\newblock URL: \url{http://www.numdam.org/item?id=PMIHES_1999__89__227_0}.

\bibitem{Lindenstrauss-Weiss-MD}
E.~Lindenstrauss and B.~Weiss.
\newblock Mean topological dimension.
\newblock {\em Israel J. Math.}, 115:1--24, 2000.
\newblock URL: \url{http://dx.doi.org/10.1007/BF02810577}, \href
  {http://dx.doi.org/10.1007/BF02810577} {\path{doi:10.1007/BF02810577}}.

\bibitem{Niu-MD-Z}
Z.~Niu.
\newblock Comparison radius and mean topological dimension: {R}okhlin property,
  comparison of open sets, and subhomogeneous {C*}-algebras.
\newblock {\em preprint}, 2019.

\bibitem{RorUHF2}
M.~R{\o}rdam.
\newblock On the structure of simple {C*}-algebras tensored with a
  {UHF}-algebra. {II}.
\newblock {\em J. Funct. Anal.}, 107(2):255--269, 1992.
\newblock URL: \url{http://dx.doi.org/10.1016/0022-1236(92)90106-S}, \href
  {http://dx.doi.org/10.1016/0022-1236(92)90106-S}
  {\path{doi:10.1016/0022-1236(92)90106-S}}.

\bibitem{Schneider-book}
R.~Schneider.
\newblock {\em Convex bodies: the {B}runn-{M}inkowski theory}, volume~44 of
  {\em Encyclopedia of Mathematics and its Applications}.
\newblock Cambridge University Press, Cambridge, 1993.
\newblock URL: \url{https://mathscinet.ams.org/mathscinet-getitem?mr=1216521},
  \href {http://dx.doi.org/10.1017/CBO9780511526282}
  {\path{doi:10.1017/CBO9780511526282}}.

\bibitem{Szabo-Z}
G.~Szab\'{o}.
\newblock The {R}okhlin dimension of topological {$\Bbb{Z}^m$}-actions.
\newblock {\em Proc. Lond. Math. Soc. (3)}, 110(3):673--694, 2015.
\newblock URL: \url{https://mathscinet.ams.org/mathscinet-getitem?mr=3342101},
  \href {http://dx.doi.org/10.1112/plms/pdu065}
  {\path{doi:10.1112/plms/pdu065}}.

\bibitem{TWW-QD}
A~Tikuisis, S.~White, and W.~Winter.
\newblock Quasidiagonality of nuclear {C*}-algebras.
\newblock {\em Ann. of Math. (2), to appear}, 09.
\newblock URL: \url{http://arxiv.org/abs/1509.08318}, \href
  {http://arxiv.org/abs/1509.08318} {\path{arXiv:1509.08318}}.

\bibitem{RC-Toms}
A.~S. Toms.
\newblock Flat dimension growth for {C*}-algebras.
\newblock {\em J. Funct. Anal.}, 238(2):678--708, 2006.

\end{thebibliography}

\end{document}